\documentclass[12pt]{amsart}
\usepackage{amsfonts}
\usepackage{graphicx}
\usepackage{amscd}
\usepackage{amsmath}
\usepackage{amssymb}
\usepackage{amsfonts}
\usepackage{tikz}
\usepackage[margin=1.2in]{geometry}
\usetikzlibrary{shapes.geometric}
\def\Nat{{\mathrm{ I\! N}}}

\tikzset{vertex/.style={shape=circle,very thin,draw,fill,inner sep=1pt,outer sep=0pt}}
\tikzset{tree/.style={shape=circle,very thin,draw,inner sep=4pt,outer sep=0pt}}

\newtheorem{theorem}{Theorem}
\theoremstyle{plain}
\newtheorem{corollary}{Corollary}

\newtheorem{lemma}{Lemma}

\theoremstyle{definition}

\numberwithin{equation}{section}

\begin{document}
\title{Asymptotic distribution of integers with certain prime factorizations}
\author[H. Vernaeve]{Hans Vernaeve}
\author[J. Vindas]{Jasson Vindas}
\author[A. Weiermann]{Andreas Weiermann}
\address{Department of Mathematics, Ghent University, Krijgslaan 281 Gebouw S22, B 9000 Gent, Belgium}
\email{hvernaev@cage.UGent.be}
\email{jvindas@cage.Ugent.be}
\email{weierman@cage.UGent.be}


\begin{abstract}
Let $p_{1}<p_2<\dots <p_{\nu}<\cdots$ be the sequence of prime numbers and let $m$ be a positive integer. We give a strong asymptotic formula for the distribution of the set of integers having prime factorizations of the form $p_{m^{k_1}}p_{m^{k_{2}}} \cdots p_{m^{k_{n}}}$ with $k_{1}\le k_{2}\le \dots \le k_{n}$. Such integers originate in various combinatorial counting problems; when $m=2$, they arise as Matula numbers of certain rooted trees. 
\end{abstract}

\maketitle

\section{Introduction}

Let $\left\{p_\nu\right\}_{\nu=1}^{\infty}$ be the sequence of all prime numbers arranged in increasing order and let $m>1$ be a fixed positive integer. We shall consider the class of integers only admitting prime factors from the subsequence $\left\{p_{m^{k}}\right\}_{k=0}^{\infty}$, that is, the set
\begin{equation}
\label{Matulaeq1}
A_{m}=\left\{p_{m^{k_1}}p_{m^{k_{2}}} \cdots p_{m^{k_{n}}}\in\mathbb{N}:\ 0\leq k_{1}\le k_{2}\leq \cdots \le k_{n}\right\}\:.
\end{equation}
The aim of this article is to provide an asymptotic formula for the distribution of $A_{m}$, namely, the following counting function
\begin{equation*}
M_{2,m}(x)=\underset{n\in A_{m}}{\sum_{n\le x}} 1\: .
\end{equation*}

The function $M_{2,2}$ arises in various interesting combinatorial counting problems; particularly, in connection with rooted trees. In 1968 Matula gave an enumeration of (non-planar) rooted trees by prime factorization \cite{Matula1968}, the so-called Matula numbers. Number theoretic aspects of this rooted tree coding have been investigated in detail in \cite{b-t,g-ivic1996}. Such numbers may be used to deduce many intrinsic properties of rooted trees \cite{Deutsch,g-ivic1994,g-y}. The set $A_{2}$ in fact corresponds to a class of Matula numbers. In Section \ref{rooted trees} we review Matula coding of rooted trees and give the interpretation of $M_{2,2}$ as the counting function of rooted trees with height less or equal to 2, under Matula's enumeration. It is worth mentioning that the significance of Matula numbers comes from applications in organic chemistry, as they can be employed to develop efficient nomenclatures for representing molecules of a variety of organic compounds
%such as alkanes, alkenes, alkines, benzenoid hydrocarbons and polyadamantenes
(cf. \cite{elk1989,elk1990,elk1994,elk2011,g-ivic-elk1993}). As explained in Section \ref{rooted trees}, $M_{2,2}$ might also be regarded as a ``transfinite counting function'' for the ordinal $\omega^{\omega}$ in a certain complexity norm \cite{Weiermann2010}. 

In \cite{Weiermann2010} Weiermann found the weak asymptotics of the counting function $M_{2,2}$. Using a Tauberian theorem by Kohlbecker for partitions \cite{kohlbecker1958}, he showed that 

\begin{equation}
\label{Matulaeq3}
\log M_{2,2}(x)\sim \pi \sqrt{\frac{2\log x}{3\log 2}}\,.
\end{equation}

The asymptotic relation (\ref{Matulaeq3}) resembles the one obtained by Hardy and Ramanujan in 1917 for the celebrated
(unrestricted) partition function, 
\begin{equation}
\label{Matulaeq4}
\log p(n)\sim \pi \sqrt{\frac{2n}{3}}\:,
\end{equation}
which they \cite{hardy-ramanujan1918}, and independently Uspensky \cite{Uspensky1920}, greatly refined later to
\begin{equation}
\label{Matulaeq5}
p(n)\sim \frac{e^{\pi \sqrt{\frac{2n}{3}}}}{(4\sqrt{3})n}\: .
\end{equation}
Naturally, the transition from (\ref{Matulaeq4}) to (\ref{Matulaeq5}) consists in finding missing asymptotic terms. The problem we address here is of similar nature. We shall fill the gap between (\ref{Matulaeq3}) and the strong asymptotics by exhibiting hidden lower order terms in the approximation (\ref{Matulaeq3}), as stated in the following theorem.

\begin{theorem} \label{Matulath1}
The function $M_{2,m}$ has asymptotic behavior
\begin{equation}
\label{Matulaeq6}
M_{2,m}(x)\sim \frac{e^{K_m}\sqrt{3}\log m}{2 \pi^{2}\log 2}(\log x)^{\frac{\log\left(\frac\pi {\sqrt{6\log m}}\right)}{2\log m}}
\exp\left(\pi\sqrt{\frac{2\log x}{3\log m}} - \frac {(\log\log x)^{2}}{8\log m}\right)\: ,
\end{equation}
where 
$$K_m=\frac{1}{2\log m}\left((\log\log m)^2+\gamma^{2}-2\gamma\log\log m-\frac{\pi^{2}}{6}-\log^2\left(\frac\pi{\sqrt{6\log m}}\right)\right)- C_{2,m}\:,$$
$\gamma$ is the Euler-Mascheroni constant, and $C_{2,m}$ is given by the convergent series
$$
C_{2,m}=\sum_{k=1}^\infty \left(\log\log p_{m^k} - \log k - \log\log m - \frac{\log k}{k \log m} - \frac{\log\log m}{k \log m}\right)\:.
$$
\end{theorem}

We will provide a proof of Theorem \ref{Matulath1} in Section \ref{proof}. The proof is based on Ingham's method from \cite{Ingham}; however, it turns out that Ingham's original theorem for partitions \cite[Thm. 2]{Ingham} is not directly applicable to our context. In Section \ref{Ingham theorem}, we shall slightly extend his result. It is likely that such an extension of Ingham's theorem might be useful for treating partition problems other than the one dealt with in this article.

\section{Two counting problems and $M_{2,m}$}
\label{rooted trees}
%We discuss in this section two counting problems that lead to consider the function (\ref{Matulaeq2}).
\subsection{Rooted trees}
\label{Matula Numbers}
Matula's coding of (non-planar) rooted trees in terms of prime factorizations provides a bijection between such trees and the positive integers. The same rooted tree enumeration was rediscovered by G\"{o}bel in \cite{gobel1980}. It is defined as follows. If we denote the trivial one-vertex tree by $\bullet$, then its Matula number is $n(\bullet):=1$. Inductively, if $T_1$, $T_2$, \dots, $T_l$ are trees and $T$ is given as
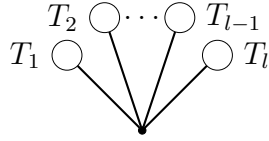
\[
\begin{tikzpicture}[thick,auto,baseline]
	\coordinate[vertex] (r) at (0,0);
	\coordinate[tree,label=left:$T_1$] (n1) at (-1,1);
	\coordinate[tree,label=left:$T_2$] (n2) at (-.5,1.5);
	\coordinate[tree,label=right:$T_{l-1}$] (n3) at (.5,1.5);
	\coordinate[tree,label=right:$T_{l}$] (n4) at (1,1);
	\coordinate[label=right:$\dots$] (dots) at (-.4,1.5);
	\draw (r) to (n1);
	\draw (r) to (n2);
	\draw (r) to (n3);
	\draw (r) to (n4);
\end{tikzpicture}
\]
then its Matula number is defined as $n(T) := p_{n(T_1)}\cdots p_{n(T_l)}$.

If $T_{1,k}$ is the tree of height one with $k$ nodes above the root, then $n(T_{1,k})= p_1^k = 2^k$. If $T$ has height two, then 
$$
n(T) = p_{n(T_{1,k_1})}p_{n(T_{1,k_2})} \cdots p_{n(T_{1,k_{\nu-1}})} p_{n(T_{1,k_\nu})} = p_{2^{k_1}}p_{2^{k_{2}}}\cdots p_{2^{k_{\nu-1}}} p_{2^{k_\nu}}\:,$$ where the $j$-th node connected to the root carries a tree $T_{1,k_j}$.
%, assumed to be ordered as $1\le k_1\le k_2\le \cdots\le k_m$.
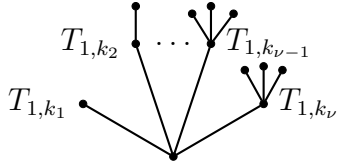
\[
\begin{tikzpicture}[thick,auto,baseline]
	\coordinate[vertex] (r) at (0,0);
	\coordinate[vertex,label=left:$T_{1,k_1}$] (n1) at (-1.2,.7);
%	\coordinate[vertex] (n11) at (-1.5,1.2);
	\coordinate[vertex,label=left:$T_{1,k_2}$] (n2) at (-.5,1.5);
	\coordinate[vertex] (n21) at (-.5,2);
	\coordinate[vertex,label=right:$T_{1,k_{\nu-1}}$] (n3) at (.5,1.5);
	\coordinate[vertex] (n31) at (.25,1.9);
	\coordinate[vertex] (n32) at (.5,2);
	\coordinate[vertex] (n33) at (.8,1.9);
	\coordinate[vertex,label=right:$T_{1,k_\nu}$] (n4) at (1.2,.7);
	\coordinate[vertex] (n41) at (.95,1.15);
	\coordinate[vertex] (n42) at (1.2,1.2);
	\coordinate[vertex] (n43) at (1.45,1.15);	
	\coordinate[label=right:$\dots$] (dots) at (-.4,1.5);
	\draw (r) to (n1);
	\draw (r) to (n2);
	\draw (r) to (n3);
	\draw (r) to (n4);
%	\draw (n1) to (n11);
	\draw (n2) to (n21);
	\draw (n3) to (n31);
	\draw (n3) to (n32);
	\draw (n3) to (n33);
	\draw (n4) to (n41);
	\draw (n4) to (n42);
	\draw (n4) to (n43);	
\end{tikzpicture}
\]
It is then clear that Matula coding gives a bijection between the set of rooted trees with height equal to 1 or 2 and the set $A_{2}$  defined in (\ref{Matulaeq1}). Consequently, $M_{2,2}(x)$ counts the number of Matula numbers corresponding to trees with $0<\text{height}(T)\le 2$, that are below $x$, i.e.,

\[\underset{\text{height}(T)\le 2}{\sum_{n(T)\le x}1}=M_{2,2}(x)+1\:.\]
Thus, this rooted tree counting function has also asymptotics (\ref{Matulaeq6}).

%It should be noticed that a similar coding with $n_{\nu}(\bullet) = \nu$ yields a rooted tree counting function that is asymptotic to $M_{2,m}(x)$ with $m=p_{\nu}$. In fact, if we denote by $n_{\nu}(T)$ integer corresponding to $T$ in this new coding and we set
%$$
%F(x)=\underset{\text{height}(T)\le 2}{\sum_{n_{\nu}(T)\le x}1}\:,
%$$ 
%then, for $x\leq\nu$,
%$$
%M_{2,m}+1\nu=
%$$

\subsection{Ordinal counting functions}
\label{ordinals} It might seem surprising at first sight that the counting function $M_{2,2}$ is related to studying asymptotic properties of 
transfinite ordinals. Since transfinite ordinals rarely show up in a number-theoretic context we will explain some features of this connection in informal
and general terms. The rest of the paper will not depend on the exposition given in this subsection, but it might be useful as a source of inspiration for further
study. 

In naive set theory ordinals generalize the ordering  of the natural numbers $0<1<2<\cdots$ by continuing beyond the first limit point $\omega$
like $0<1<2<\cdots<\omega+1<\omega+2<\cdots$. This process can be continued beyond the next limit $\omega+\omega$ like 
$0<1<2<\cdots<\omega+1<\omega+2<\cdots< \omega\cdot 2 +1<\omega\cdot 2 +2<\cdots$ and by iteration like
$0<1<2<\cdots<\omega+1<\omega+2<\cdots< \omega\cdot 2 +1<\omega\cdot 2+2 <\cdots<\omega\cdot 3<\cdots<\omega\cdot n<\cdots$.
At a certain moment we reach the first limit of limits $\omega\cdot \omega$ and, again by iteration, we reach limits of limits of limits and in the limit of this counting
we reach $\omega^\omega$ (an ordinal which -- as will become clear soon -- is of relevance to $M_{2,m}$).

The ordinal $\omega^\omega$ is not at all frightening since it appears as the order type of the polynomials in $\Nat[x]$ under eventual domination
or as the order type of the multisets of natural numbers. There is of course no bound in counting through the ordinals and by further counting we reach
$\omega^{\omega^\omega}, \omega^{\omega^{\omega^\omega}},\ldots$, but the higher we go the more complicated the description becomes.
Slight extensions could still be dealt with by combinatorial means (which can still be formalized in Peano arithmetic) and stronger extensions will require
from some moment onwards basic set theoretic machinery.

There is still some nice and accessible visualization of the ordinals less than $\varepsilon_0$, which is the limit of the finite powers of $\omega$
showing up in the sequence $\omega,\omega^\omega,\omega^{\omega^\omega}, \omega^{\omega^{\omega^\omega}},\ldots$. For this we consider a subclass
of Hardy's orders of infinity. Let ${\mathcal E}$ be the class of unary functions $f:\Nat\to\Nat$ such that
\begin{enumerate}
\item the function $c_0$ is an element of ${\mathcal E}$ where $c_0(x)=x$ and
\item with two functions $f,g$ in ${\mathcal E}$ also the function $h$ is in ${\mathcal E}$, where $h(x)=x^{f(x)}+g(x)$.
\end{enumerate}
On ${\mathcal E}$ we define the ordering of eventual domination as usual by $f\prec g$ if and only if there exists a non-negative integer $k$ such that $f(x)<g(x)$ for all $x\geq k$.
The structure $\langle {\mathcal E},\prec \rangle $ is isomorphic with $\langle \{\alpha:\alpha<\varepsilon_0\},< \rangle $ and so we can identify both structures.
If we also write $id=c_{0}$ for the identity function on $\Nat$, then the isomorphism maps $\omega$ to $id$, $\omega^\omega$ to $id^{id}$, $\omega^{\omega^\omega}$ to $id^{id^{id}}$,
etc. The ordinal $\varepsilon_0$ is the proof-theoretic ordinal of first order Peano arithmetic $PA$. $PA$ proves (after an appropriate formalization of the context)
the scheme of transfinite induction for all strict initial segments of $\varepsilon_0$ but not the scheme of transfinite induction for the full segment up to $\varepsilon_0$.

For treating ${\mathcal E}$ in the context of arithmetic we need a specific (easily definable) coding of the elements of ${\mathcal E}$ into the natural numbers.
One of the standard devices for achieving this is provided by associating to the elements of ${\mathcal E}$ their canonical counterparts in the finite non-planar rooted
trees. Such a bijection $t$ can defined recursively as follows. First, $t(c_0):=\bullet$. Every $f$ can be written as $f=id^{g_1}+\cdots+id^{g_n}$, then let
$$t(f):=
\begin{tikzpicture}[thick,auto,baseline]
	\coordinate[vertex] (r) at (0,0);
	\coordinate[tree,label=left:$t(g_1)$] (n1) at (-1,1);
	\coordinate[tree,label=left:$t(g_2)$] (n2) at (-.5,1.5);
	\coordinate[tree,label=right:$t(g_{l-1})$] (n3) at (.5,1.5);
	\coordinate[tree,label=right:$t(g_{l})$] (n4) at (1,1);
	\coordinate[label=right:$\dots$] (dots) at (-.4,1.5);
	\draw (r) to (n1);
	\draw (r) to (n2);
	\draw (r) to (n3);
	\draw (r) to (n4);
\end{tikzpicture}
.$$

By this identification we can --  using Matula's coding -- canonically associate to $f\in {\mathcal E}$ its G\"odel number $\lceil f\rceil:=n(t(f))$.
This coding has been used explicitly by Troelstra and Schwichtenberg in \cite[p. 320, Def. 10.1.5]{Troelstra}.
In this context we arrive at the following interpretation
$$1+M_{2,2}(x)=\#\{f\in {\mathcal E}: f\prec id^{id}\wedge \lceil f\rceil\leq x\}\: .$$

For coding a larger segment of ordinals Sch\"utte \cite[Sec. V.8]{Schutte} used a related coding $Nr$ which
when restricted to ${\mathcal E}$ has the property that $1+M_{2,4}(x)$ is the number of ordinals $\alpha$ below $\omega^\omega$
such that $Nr(\alpha)\leq x$.

Until now, the study of ordinal counting functions has found applications to logical limit laws for ordinals and to phase transitions for G\"odel incompleteness
results (it seems very interesting and intriguing to find additional applications). A further discussion of phase transitions will be beyond the scope of this exposition,
but we want to include an intriguing example for a zero-one law (see \cite{burris} for an account on logical limit laws).
As usual, we use $\models$ for the satisfaction relation from model theory.
Let $\varphi$ be a sentence in the language of linear orders. 
Let $$\delta_\varphi:=\lim_{x\to\infty}\frac{\#\{f\in {\mathcal E}:  f\prec id^{id} \wedge \lceil f\rceil\leq x \wedge \langle \{g\in {\mathcal E}:g\prec f \},\prec \rangle \models \varphi \}}
{\#\{f\in {\mathcal E}:  f\prec id^{id} \wedge \lceil f\rceil\leq x\}}\: .$$
Then $\delta_\varphi$ exists and either $\delta_\varphi=1$ or $\delta_\varphi=0.$
A proof of this and similar results has been obtained in \cite{Weiermann2012} by an analysis of the asymptotic behavior of $M_{2,2}$ and related counting functions.

At the beginning of this subsection it has been indicated that ordinals might provide a source of inspiration for further research
and we will now indicate some possible options.
For $g\in {\mathcal E}$, let $c_g(x):={\#\{f\in {\mathcal E}:  f\prec g \wedge \lceil f\rceil\leq x\}}$. For various choices
of $g$ some preliminary results on weak asymptotics for $c_g$ have been obtained in \cite{Weiermann2010}.
Moreover, strong asymptotics for $c_{{id^k}}$ can be obtained by elementary means.

We believe that the methods of this paper will allow one to provide strong asymptotics for
$c_{id^{id^k}}$ for any given fixed $k$. A strong asymptotic formula for $c_{id^{id^{id}}}$ (which would resemble something like
multiplicative double partitions) seems however to require new methods. A general challenge would be then to provide a general theorem
on strong asymptotics for $c_g$ for any fixed $g$ and for analogous functions emerging from the Sch\"utte coding.

\section{An extension of Ingham's theorem for unrestricted partitions}
\label{Ingham theorem}
As mentioned in the Introduction, we need an extension of Ingham's theorem for strong asymptotics of partition functions. The extension will follow from a complex Tauberian theorem for large asymptotic behavior of the Laplace transform, also due to Ingham \cite[Thm. $1'$]{Ingham}.

Let $0<\lambda_{0}<\lambda_1<\dots<\lambda_k\to \infty$ be a sequence of real numbers and let
\[N(u)= \sum_{\lambda_k\le u}1\]
be its counting function. Consider the additive semigroup $\Lambda$ generated by $\left\{\lambda_k\right\}_{k=0}^{\infty}$, i.e.,
\[\Lambda = \{r\in\mathbb{R}: r=\sum_{k=0}^l n_k\lambda_k,\ n_k\in\mathbb{N}
\}\:.\]
For $r\in\Lambda$, the partition function $p(r)$ is defined as the number of ways of writing $r$ as $r=\sum_{k=0}^l n_k\lambda_k$.
We further set
\[P(u)= \underset{r\in\Lambda}{\sum_{r\le u}} p(r)\:.\]
The following theorem obtains the asymptotic behavior of $P(u)$ if one knows a certain average asymptotic behavior for $N(u)$. It slightly extends that of Ingham by allowing an extra term of the form $B\log^2 u$ in the asymptotic expansion (\ref{eqN}). As usual, $\zeta$ stands for the Riemann zeta function and $\Gamma$ for the Euler Gamma function.  The constant $\gamma_{1}$ denotes the Stieltjes constant, that is,
$$
\gamma_{1}=\lim_{n\to\infty} \sum_{k=1}^n \frac{\log k}k-\frac{\log^2 n}2\: .
$$ 
\begin{theorem}\label{main-partition-thm}
Suppose that
\begin{equation}
\label{eqN}
\int_0^u \frac{N(t)}t\,dt = \frac{A}{\alpha}u^\alpha + B\log^2 u + C \log u + D +o(1)\: ,
\end{equation}
with $\alpha,A>0$. Then
\begin{equation}
\label{Matulaeq3.2}
P(u) \sim \left(\frac{1-\beta}{2\pi}\right)^{\frac12} e^{D'} M^{-(C+\frac12)} u^{C-\beta C-\frac\beta2}\exp\left(\frac{Mu^\beta}{\beta}+ B\log^2\left(\frac{u^{1-\beta}}{M}\right)\right)\:,
\end{equation}
where
\[\beta=\frac\alpha{\alpha+1},\quad M=(A\alpha\Gamma(\alpha+1)\zeta(\alpha+1))^{\frac1{\alpha+1}},\quad D'= D+ \left(\frac{\pi^{2}}{6}-2\gamma_1 -\gamma^2\right)B.\]
\end{theorem}

In order to deduce Theorem \ref{main-partition-thm} from Ingham's Tauberian theorem, we proceed to find the asymptotic behavior of the Laplace-Stieltjes transform of $P$. Set
$$
F(s)=\sum_{r\in \Lambda}e^{-sr}=\int_{0}^{\infty}e^{-su}dP(u)
$$
and 
$$f(s)=s\int_{0}^{\infty}\frac{N(u)}{e^{su}-1}\:du\:.$$
The generating function identity $F(s)=e^{f(s)}$ is well-known.
\begin{lemma}
\label{Ingham-extl1}
If $(\ref{eqN})$ holds, then, as $\sigma\to 0^{+}$,
$$
\sigma\int_0^\infty \frac{N(u)}{e^{\sigma u}-1}\,du = \frac{A\Gamma(\alpha+1)\zeta(\alpha+1)}{\sigma^\alpha} + B\log^2 \sigma - C\log \sigma + D' + o(1)\: .
$$
Here 
\begin{equation}
\label{Matulaeq3.3}
D'= D+ \left(\frac{\pi^{2}}{6}-2\gamma_1 -\gamma^2\right)B\:.
\end{equation}
\end{lemma}
\begin{proof}
We employ standard Schwartz distribution calculus in our manipulations. It might also be possible to give a classical proof along the lines of that of \cite[Thm.\ IV.23.1]{Korevaarbook}. For Schwartz distributions, we follow the notation exactly as in \cite[Chap. 2]{estrada-kanwal2002}. Taking distributional derivative in (\ref{eqN}), we obtain
\begin{align*}
\frac{N(\lambda u)}{u}=&A\lambda^{\alpha}u_{+}^{\alpha-1}+(B\log^{2}\lambda+C\log\lambda+D)\delta(u)+2B\operatorname*{Pf}\left(\frac{H(u)\log u}{u}\right)\\
& +(2B\log\lambda+C)\operatorname*{Pf}\left(\frac{H(u)}{u}\right)+o(1)\:, \ \ \ \lambda\to\infty\:,
\end{align*}
distributionally in the space of tempered distributions $\mathcal{S}'$ (cf. \cite[Sec. 3.9]{estrada-kanwal2002}, \cite[Sec. 2.5]{p-s-v}), where $\delta$ stands for the Dirac delta distribution, $H$ is the Heaviside function, and $\operatorname*{Pf}$ denotes regularization via Hadamard finite part \cite[Sec. 2.4]{estrada-kanwal2002}. Testing this asymptotic expansion at the test function $\psi(u)=u/(e^u-1)$, setting $\sigma=1/\lambda$, and taking into account the well-known formula
\begin{equation}
\label{eqRiemann}
\Gamma(s+1)\zeta(s+1)=\int_{0}^{\infty}\frac{u^{s}}{e^{u}-1}\: du\: , \ \ \ \Re e\:s>0\:,
\end{equation}
we obtain
$$
\sigma\int_0^\infty \frac{N(u)}{e^{\sigma u}-1}\,du=\frac{A\Gamma(\alpha+1)\zeta(\alpha+1)}{\sigma^\alpha} + B\log^2 \sigma - \left(C+2BK\right)\log \sigma + D' + o(1)\: ,
$$
where $D'=D+2BK'+CK$, and the constants $K$ and $K'$ are given by the Hadamard finite part at $0$ of the integrals
$$
K=\operatorname*{F.p.}\int_{0}^{\infty}\frac{du}{e^{u}-1} \quad \mbox{and} \quad K'=\operatorname*{F.p.}\int_{0}^{\infty}\frac{\log u}{e^{u}-1}\:du\:.
$$
Hence, it remains to evaluate these two constants. We will do so by inspecting the Laurent expansion at $s=0$ of the analytic continuation of (\ref{eqRiemann}). In fact, the classical procedure of Marcel Riesz \cite{estrada-kanwal2002,gelfand-shilovVol1} yields the analytic continuation of (\ref{eqRiemann}) to $\mathbb{C}\setminus\left\{0,-1,-2,\dots\right\}$ as the finite part integral of the right hand side. By employing the Gelfand-Shilov Laurent expansion at $s=0$ for the distribution $u_{+}^{s-1}$ \cite[p. 87]{gelfand-shilovVol1}, we conclude
$$
\operatorname*{F.p.}\int_{0}^{\infty}\frac{u^{s}}{e^{u}-1}\: du= \frac{1}{s}+\sum_{n=0}^{\infty}\frac{s^{n}}{n!}\operatorname*{F.p.}\int_{0}^{\infty}\frac{\log^{n}u}{e^{u}-1}\:du\,, \quad 0<|s|<1\:.
$$  
On the other hand, since
$$\zeta(s+1)=\frac{1}{s}+\gamma-\gamma_{1} s+\cdots\quad \mbox{and} \quad \Gamma(s+1)=1-\gamma s+\left(\frac{\gamma^{2}}{2}+\frac{\pi^{2}}{12}\right)s^{2}+\cdots\:,$$
we have
$$
\Gamma(s+1)\zeta(s+1)=\frac{1}{s} +\left(\frac{\pi^{2}}{12}- \gamma_{1}-\frac{\gamma^{2}}{2}\right)s+ \cdots\:, \quad 0<|s|<1\:,
$$
and therefore $K=0$ and $K'=\pi^{2}/12- \gamma_{1}-\gamma^{2}/2$.
\end{proof}

\begin{proof}[Proof of Theorem \ref{main-partition-thm}]
We can now apply Ingham's Tauberian theorem \cite[Thm. $1'$]{Ingham} to the generating function $F(s)$
with
\[\varphi(s):= \frac{A\Gamma(\alpha+1)\zeta(\alpha+1)}{s^\alpha} \quad\text{and}\quad \chi(s):= e^{D'} s^{-C}e^{B\log^{2}s}\:.\]
Indeed, in view of Lemma \ref{Ingham-extl1},
$$
F(\sigma)\sim \chi(\sigma)e^{\varphi(\sigma)}, \quad \sigma\to0^{+}\:,
$$
and the quoted theorem of Ingham immediately implies that 
$$
P(u)\sim \frac{\chi(\sigma(u))\exp\left(\varphi(\sigma(u))+u\sigma(u)\right)}{\sqrt{2\pi\sigma^{2}(u)\varphi''(\sigma(u))}}, \quad u\to\infty\:,
$$
where $\sigma(u)$ is the inverse function of $-\varphi'$, i.e., $\sigma(u)=M u^{-\frac1{\alpha+1}}$ with $M=(A\alpha\Gamma(\alpha+1)\zeta(\alpha+1))^{\frac1{\alpha+1}}$. Setting $\beta=\alpha/(\alpha+1)$, we have
\begin{align*}
e^{\varphi(\sigma(u))+ u\sigma(u)} &= e^{\frac{Mu^\beta}\beta}, \\
\frac{\chi(\sigma(u))}{\sigma(u)\sqrt{2\pi \varphi''(\sigma(u))}} &= \left(\frac{1-\beta}{2\pi}\right)^{\frac12} e^{D'} M^{-(C+\frac12)} u^{C-\beta C-\frac\beta2}e^{B\log^2(u^{1-\beta}/M)}\,.
\end{align*}
whence $(\ref{Matulaeq3.2})$ follows.
\end{proof}

In the sequel, we will only use the case $\alpha=1$ of Theorem \ref{main-partition-thm}, which we state in the next corollary for the sake of convenience.
\begin{corollary}\label{cor-partition-thm}
Suppose that
\[\int_0^u \frac{N(t)}t\,dt = Au + B\log^2 u + C\log u + D + o(1)\:.\]
Then
\[P(u) \sim \frac{e^{D'+B\log^2\left(\pi\sqrt{\frac A6}\right)}}{2\sqrt \pi} {\left(\pi\sqrt{\frac A6}\right)}^{-(C+\frac12)} u^{\frac C2 - \frac14 -B\log\left(\pi\sqrt{\frac A6}\right)}\exp\left(\pi\sqrt{\frac{2Au}3}+ \frac B4 \log^2 u\right)\:,\]
where $D'$ is given by (\ref{Matulaeq3.3}).
\end{corollary}

\section{The original problem}
\label{proof}
We now proceed to give a proof of Theorem \ref{Matulath1}. We translate our original problem into an additive partition problem. Consider $\lambda_k = \log p_{m^k}$. Then, with the notation of the preceding section,
\[1+M_{2,m}(e^u)  %\sum_{\scriptstyle\sum_{j=1}^m n_j \log p_{l^j}\le u\atop n_j\ge 0} 1 
= \underset{r\in\Lambda}{\sum_{r\le u}} p(r) = P(u)\: .\]
Of course, here $p(r)=1$ for each $r\in \Lambda=\left\{\sum_{k=0}^{l}n_{k}\log p_{m^{k}}:\:n_{k}\in\mathbb{N} \right\}$.

Thus, we are interested in the average asymptotics of the counting function 
$$N(u)=\sum_{\log p_{m^k}\le u} 1\, .$$ 
Observe that
\begin{equation}\label{integral-N-over-t}
\int_0^u \frac{N(t)}t\,dt% = N(u)\log u - \int_0^u \log t\, dN(t)
= N(u)\log u - \sum_{\log p_{m^k}\le u}\log\log p_{m^k}\:.
\end{equation}
We first need to estimate $\log\log p_{m^k}$.

\begin{lemma}
\label{Matulal2}
%\[\log p_{l^k} = k\log l + \log k + \log\log l + O\left(\frac{\log k}{k}\right)\]
%and
\begin{equation}\label{loglogplk}
\log\log p_{m^k} = \log k + \log\log m + \frac{\log k}{k \log m} + \frac{\log\log m}{k \log m} + O\left(\frac{\log^{2} k}{k^2}\right).\end{equation}
\end{lemma}

\begin{proof}
Using the prime number theorem, Cipolla \cite{Cipolla} found in 1902 an asymptotic formula for $p_n$. Employing just two terms in the expansion, we have
\[p_n = n\log n + O(n\log\log n).\]
This leads to
\[\log p_n = \log n + \log\log n + O\left(\frac{\log\log n}{\log n}\right)\]
and
\[\log\log p_n = \log\log n + \frac{\log\log n}{\log n} + O\left(\frac{(\log\log n)^2}{\log^2 n}\right)\:.\]
Thus, for $n=m^k$, we obtain the required formula.
\end{proof}
Lemma \ref{Matulal2} immediately yields:
\begin{corollary}
\begin{equation}\label{C2l}
C_{2,m}=\sum_{k=1}^\infty \left(\log(\log p_{m^k}) - \log k - \log\log m - \frac{\log k}{k \log m} - \frac{\log\log m}{k \log m}\right)
\end{equation}
converges.
\end{corollary}

Next,

\begin{lemma}\label{avg-asymptotics-of-N}
\begin{equation}\label{Matulaeq4.4}
\int_0^u \frac{N(t)}{t}\,dt 
=\frac{u}{\log m} - \frac{\log^2 u}{2\log m} + \frac{1}{2}\log u  +D_m + o(1)\: ,
\end{equation}
where 
$$D_m = \frac{(\log\log m)^2}{2\log m}  + \log\left(\frac{1}{\log 2}\sqrt{\frac{\log m}{2\pi}}\right)- C_{2,m} - \frac{\gamma_1}{\log m} - \frac{\log\log m}{\log m}\gamma\:.$$
\end{lemma}

\begin{proof}
We first notice that
\[
p_{m^k}\le e^u 
%\mbox{ iff } l^m\le \pi(e^u) 
\mbox{ if and only if } k\le y:= \frac{\log(\pi(e^u))}{\log m}\:,
\]
where $\pi(x)$ is the distribution of the prime numbers. By the prime number theorem,
\[y %= \frac{1}{\log l}\log \left(\frac{e^u}{u} + O\left(\frac{e^u}{u^2}\right)\right)
 = \frac u {\log m} - \frac{\log u}{\log m} + O\left(\frac 1 u\right)\:.
\]
Thus $N(u)=\left\lfloor  y \right\rfloor+1$ and, combining equations\ \eqref{integral-N-over-t}, \eqref{loglogplk} and \eqref{C2l},
\begin{align*}
\int_0^u \frac{N(t)}{t}\,dt 
&= (\left\lfloor y \right\rfloor+1) \log u -\log\log 2-\sum_{k=1}^{\left\lfloor y \right\rfloor} \log(\log p_{m^k})
\\
&=(\left\lfloor  y \right\rfloor +1)\log u - \log\log2- C_{2,m} - \log(\left\lfloor  y \right\rfloor !) - (\log\log m)\left\lfloor  y \right\rfloor 
\\
&
\ \ \ - \frac{1}{\log m}\sum_{k=1}^{\left\lfloor  y \right\rfloor} \frac{\log k}k - \frac{\log\log m}{\log m}\sum_{k=1}^{\left\lfloor  y \right\rfloor} \frac{1}k + o(1)\:.
\end{align*}

Using Stirling's formula and the defining formulas for $\gamma$ and $\gamma_1$
\[
\log(n!) = n\log n - n + \frac{1}{2}\log n + \log(\sqrt{2\pi}) + o(1)\:,
\]
\[
\gamma = \sum_{k=1}^n \frac{1}{k} - \log n + o(1)
\quad\text{and}\quad
\gamma_1 = \sum_{k=1}^n \frac{\log k}k - \frac{\log^2 n}2 + o(1)\:,
\]
we have
\begin{align*}
\int_0^u \frac{N(t)}{t}\,dt 
&=\left\lfloor  y \right\rfloor \log \left(\frac{u}{\left\lfloor  y \right\rfloor \log m}\right) +\log u+ \left\lfloor  y \right\rfloor - \frac{\log^2\left\lfloor  y \right\rfloor}{2\log m} - \left(\frac{1}{2} + \frac{\log\log m}{\log m}\right)\log\left\lfloor  y \right\rfloor\\
& \ \ \ -\log\log 2- C_{2,m} - \log(\sqrt{2\pi}) - \frac{\gamma_1}{\log m} - \frac{\log\log m}{\log m}\gamma + o(1)\:.
\end{align*}
Since
\[
\left\lfloor  y \right\rfloor \log \left(\frac{u}{\left\lfloor  y \right\rfloor\log m}\right)
=\left\lfloor  y \right\rfloor \log \left(1 + \frac{u-\left\lfloor  y \right\rfloor \log m}{\left\lfloor  y \right\rfloor \log m}\right)
= \frac{u-\left\lfloor  y \right\rfloor \log m}{\log m} + O\left(\frac{\log^{2}u} {u}\right)
\]
and
\[
\log\left\lfloor  y \right\rfloor= \log (y+ O(1)) = \log y + O\left(\frac1y\right) = \log u - \log\log m+ O\left(\frac{\log u}u\right),
\]
we obtain (\ref{Matulaeq4.4}), as required.
\end{proof}

The asymptotic formula (\ref{Matulaeq6}) follows by combining Lemma \ref{avg-asymptotics-of-N} and Corollary \ref{cor-partition-thm} after a straightforward calculation. The proof of Theorem \ref{Matulath1} is complete.

\end{document}